\documentclass[letterpaper,12pt, reqno]{amsart}
\usepackage{latexsym, amsmath, amssymb, amsthm, mathrsfs}
\usepackage{fullpage,url,amssymb,amsmath,graphicx,color, mathrsfs}
\usepackage{bbm}
\usepackage{url}
\usepackage{amsthm}
\usepackage{amsrefs}
\usepackage{amsfonts}
\usepackage{subcaption}
\usepackage{hyperref}
\usepackage[ruled,noend,noline,slide]{algorithm2e}
\usepackage[table,xcdraw]{xcolor}

\newtheorem{lemma}{Lemma}[section]
\newtheorem{theorem}[lemma]{Theorem}
\newtheorem*{theorem*}{Theorem}
\newtheorem{cor}[lemma]{Corollary}
\newtheorem{proposition}[lemma]{Proposition}
\newtheorem{claim}[lemma]{Claim}

\newtheorem{claim*}{Claim}
\newtheorem{remark}[lemma]{Remark}

\newcommand{\Var}{\operatorname{Var}}

\newcommand{\calA}{{\mathcal A}}

\newcommand{\var}{\operatorname{Var}}

\newcommand{\og}{\mbox{O}_{\Pr}}

\newcommand{\och}{\mbox{o}_{\Pr}}

\newcommand{\beq}{\begin{equation}}
\newcommand{\enq}{\end{equation}}
\newcommand{\beqa}{\begin{eqnarray}}
\newcommand{\enqa}{\end{eqnarray}}
\newcommand{\neno}{\newline\noindent}

\newcommand{\dds}{\,\mbox{d}}

\newcommand{\PP}{\mathbb{P}}

\newcommand{\RR}{\mathbb{R}}
\newcommand{\NN}{\mathbb{N}}
\newcommand{\real}{\mathbb{R}}

\newcommand{\EE}{\mathbb{E}}
\newcommand{\esp}{\mathbb{E}}
\DeclareMathOperator*{\argmin}{arg\,min}
\DeclareMathOperator*{\argmax}{arg\,max}

\newenvironment{Ualgorithm}[1][htpb]
  {\def\@algocf@post@ruled{\kern\interspacealgoruled\hrule  height\algoheightrule\kern3pt\relax}%
    \def\@algocf@capt@ruled{under}
    \begin{algorithm}[#1]}
  {\end{algorithm}}

\title{Local angles and dimension estimation\\from data on manifolds}
\date{}
\begin{document}
\author{Mateo D\'iaz}
\address{Mateo D\'iaz,
Center for Applied Mathematics,
657 Frank H.T. Rhodes Hall,
Cornell University,
Ithaca, NY 14853} 
\email{md825@cornell.edu }

\author{Adolfo J. Quiroz}
\address{Adolfo J. Quiroz, Departamento de
  Matem\'aticas\\ Universidad de los Andes\\ Carrera 1 No. 18a 10\\ Edificio
  H\\ Primer Piso\\ 111711 Bogot\'a\\ Colombia}
\email{aj.quiroz1079@uniandes.edu.co}
\author{Mauricio Velasco}
\address{Mauricio Velasco, Departamento de
  Matem\'aticas\\ Universidad de los Andes\\ Carrera 1 No. 18a 10\\ Edificio
  H\\ Primer Piso\\ 111711 Bogot\'a\\ Colombia}
\email{mvelasco@uniandes.edu.co}


\begin{abstract} For data living in a manifold $M\subseteq \RR^m$ and a point $p\in M$ we 
consider a statistic $U_{k,n}$ which estimates the variance of the angle between pairs of 
vectors $X_i-p$ and $X_j-p$, for data points $X_i$, $X_j$, near $p$, and evaluate this statistic as a tool for estimation of the intrinsic dimension of $M$ at $p$. Consistency of the local dimension
estimator is established and the asymptotic distribution of $U_{k,n}$ is found under minimal
regularity assumptions. Performance of the proposed methodology is compared against state-of-the-art methods on simulated data. 
 
\end{abstract}

\maketitle
\section{Introduction}\label{sec:1}
Understanding complex data sets often involves dimensionality reduction. This is particularly necessary in the analysis of images, and when dealing with genetic or text data. Such data sets are usually presented as collections of vectors in $\real^m$ and it often happens that there are non-linear dependencies among the components of these data vectors. In more geometric terms these non-linear dependencies amount to saying that the vectors lie on a submanifold $M\subseteq \real^m$ whose dimension $d$ is tipically much smaller than $m$. The expression {\it manifold learning} has been coined in the literature for the process of finding properties of $M$ from the data points. 

Several authors in the artificial intelligence literature have argued about the convenience of having methods to find or approximate these low-dimensional manifolds \cite{bn,fsa,rs,sbn,srh,ten}. Procedures for achieving this kind of low dimensional representation are called {\it manifold projection methods}.
Two fairly successful such methods are 
Isomap of Tenenbaum, de Silva and Langford \cite{ten} and the Locally
Linear Embedding method of Roweis and Saul \cite{rs}. For these and other
manifold projection procedures, a key initial ingredient is a  precise estimation of the integer $d$, ideally obtained at low computational cost.

The problem of estimating $d$ has been the focus of much work in statistics starting from the pioneering work of Grassberger-Procaccia~\cite{gp}.  Most of the most recent dimension identification procedures appearing in the literature are either related to graph theoretic ideas \cite{bqy,bqy13,y98,py01,py11} or to nearest neighbor distances \cite{pbjd,lb,cgh,lombardi11}. A key contribution of the latter group is the work of Levina and Bickel \cite{lb} who propose a
``maximum likelihood'' estimator of intrinsic dimension. To describe it let $L_k(X_i)$ be the distance from the sample point $X_i$ to its $k$-th nearest neighbor in the sample (with respect to the euclidean distance in the ambient space $\real^m$). Levina and Bickel show that, asymptotically, the expected value of the statistic
\beq\label{e1}
\widehat
m_k(X_i):=\left[\frac{1}{k-2}\sum_{j=1}^{k-1}
\log\frac{L_k(X_i)}{L_j(X_i)}\right]^{-1}
\enq 
coincides with the intrinsic dimension $d$ of the data. As a result, they propose the corresponding 
sample average $\overline m_k:= n^{-1} \sum_{i=1}^n \widehat
m_k(X_i)$ as an estimator of dimension. Asymptotic properties of this statistic have
been obtained in the literature (see \cite[Theorem 2.1]{py11}) allowing for the construction of confidence intervals.  Both the asymptotic expected value and the asymptotic distribution are independent of the underlying density from which the sample points are drawn and thus lead to a truly non-parametric estimation of dimension.

In addition to distances, Ceruti et al. propose in~\cite{ceruti-D} that angles should be incorporated in the dimension estimators. This proposal, named \texttt{DANCo}, combines the idea of norm concentration of 
nearest neighbors with the idea of angle concentration for pairs of points on
the $d$-dimensional unit sphere. 

The resulting dimension identification procedure is relatively involved. The method combines two ideas. On one hand it uses the Kullback-Leibler divergence to measure the distance between the estimated probability density function (pdf) of the normalized nearest neighbor distance for the data considered and the corresponding pdf of the distance from the center of an $r${-}dimensional unit ball to its nearest neighbor under uniform sampling. On the other hand, it uses a concentration result due to S\"odergren \cite{soderg}, for angles corresponding to {\it independent} pairs of points on a sphere.

The main contribution of this article is a new and simple dimension identification procedure based solely on angle concentration. We define a $U$-statistic which averages angle squared deviations over all pairs of vectors in a nearest neighbor ball of a fixed 
point and determine its asymptotic distribution. In the basic version of our proposed method there is no need of calibration of distributions and moreover our statistic is a $U$-statistic among {\it dependent} pairs of data points and it is well known that these offer fast convergence to their mean and asymptotic distribution.

Our method has been called ANOVA in the literature\footnote{The term was coined by Breiding, 
Kalisnik, Sturmfels and Weinstein  in~\cite{BKSW} when describing an earlier preliminary version of 
this article}, given
 that the $U$-statistic used, $U_{k,n}$ to be defined below, is an estimator 
 of the variance of the angle between pairs of vectors among uniformly chosen points in the sphere 
 $S^{d-1}$. Our main results are to prove the consistency of the proposed method of 
 estimation (Proposition \ref{prop: consist}) 
 and the description of the (suitably normalized) asymptotic 
 distribution of the statistic considered (Theorem~\ref{thm:main}), a result that is very useful in the construction of asymptotic confidence intervals in dimension estimation. We describe our proposed method in Section~\ref{sec:2} and provide its theoretical 
 justification in Section~\ref{sec:3}. Sections~\ref{sec:4} and \ref{sec:5} discuss the details of our implementation of the dimension identification 
 procedure together with some empirical improvements. It also contains the result of performance evaluations on simulated examples, 
 including comparisons with current state-of-the-art methods.

\section{A $U$-statistic for dimension identification}\label{sec:2}
\subsection{Description of the statistic}
Suppose our data form an i.i.d. sample, $X_1,\dots, X_n$ from a distribution $P$ on $\real^m$
with support on a Riemannian $C^2$ manifold $M$ of dimension $d<m$. Given a point $p\in M$, the
question to be addressed is to determine the dimension $d$ of the tangent space of $M$ at $p$ using 
only information from sample points near $p$ (we want to allow for the value of $d$ to depend on the point $p$ and for $M$ to be disconnected). 

The simplest version of our dimension identification procedure is described by the following steps: 
\begin{enumerate}
\item For an appropriate value of the constant $C$, to be specified below, let
$k:=\lceil C\log(n)\rceil$. Assume, relabeling the sample if necessary, that $X_1,\dots, X_k$ are the 
$k$ nearest neighbors of $p$ in the sample, according to the euclidean distance in $\RR^m$.
\item Define the {\it angle-variance} $U$-statistic, $U_{k,n}$, by the formula
\beq\label{def:statistic}
U_{k,n}:=\frac{1}{\binom{k}{2}} \sum_{1\leq i<j\leq k} \left(\arccos\left\langle \frac{X_i-p}{\|X_i-p\|},\frac{X_j-p}{\|X_j-p\|}\right\rangle-\frac{\pi}{2}\right)^2,
\enq
where $\langle\cdot,\cdot\rangle$ denotes the dot product on $\real^m$.

\item  Estimate the unknown dimension $d$ as $\widehat d$, equal to the integer $r$ such that $\beta_{r}$ is closest to $U_{k,n}$, for a sufficiently large sample size $n$, where $\beta_r$ is the quantity defined by 
\beq\label{beta_d}
\beta_r:=
\begin{cases}
\frac{\pi^2}{4}-2\sum_{j=0}^s \frac{1}{(2j+1)^2}\text{ if $r-2=2s+1$ is odd or}\\
\frac{\pi^2}{12} -2\sum_{j=1}^s\frac{1}{(2j)^2}\text{ if $r-2=2s$ is even.}

\end{cases}
\enq
\end{enumerate}
The key idea of our estimator goes as follows: 
For large $n$ and the chosen value of $k$, the nearest neighbors
of $p$ in the data set, behave as uniform data on a small ball around $p$ in the embedded tangent space of $M$ at this point, and the corresponding unit vectors, $(X_i-p)/\|X_i-p\|$, are nearly
uniform on the unit sphere of the tangent space, $S^{d-1}$. For uniform data on $S^{d-1}$,
the expected angle between two random vectors is always $\pi/2$ (regardless of $d$), but
the variance of this angle decreases rapidly with $d$. Formula (\ref{beta_d}) gives the
value of this variance for every dimension $r$. Since our results below show that
the $U$-statistic, $U_{k,n}$, will converge in probability to $\beta_d$ for the actual
dimension of $M$ at $p$, estimation of $d$ by choosing the $r$ such that
$\beta_r$ closest to $U_{k,n}$ will be consistent.
An additional fact that helps in this convergence is that the variance of $U_{k,n}$,
which depends on the fourth moment of the angles, is also converging rapidly to zero.

The following subsection establishes useful facts about angles between 
random points on the unit sphere $S^{d-1}$ of $\real^d$ 
and, in particular, about moments of the function
\beq\label{kernel}
h(z,z')=\left(\arccos\left\langle z,z'\right\rangle-\frac{\pi}{2}\right)^2
\enq
when computed on data uniformly distributed on $S^{d-1}$.
Section~\ref{sec:3}, building on subsection \ref{subsec: part3}, develops
the theoretical results that serve as basis for the use
of $U_{k,n}$ on manifolds.

\subsection{Angle-variance statistics for pairs of uniform points on $S^{d-1}$}
\label{subsec: part3}

\begin{lemma}[\bf Angles between uniform vectors] \label{lem: anglePairs} Let $Z_1,Z_2$ be two independent vectors with the uniform distribution on the unit sphere $S^{d-1}\subseteq \RR^d$ and let $\Theta_d:=\arccos\langle Z_1,Z_2\rangle$ be the angle between them. The following statements hold:
\begin{enumerate}
\item The distribution of $\Theta_d$ is given by
\[\PP(\Theta_d\leq \alpha)= \frac{\int_0^{\alpha}\sin^{d-2}(\phi)d\phi}{\int_0^{\pi} \sin^{d-2}(\phi) d\phi}.\]

\item The moment generating function of $\Theta_d$, denoted by $\phi_{d-2}(s):=\EE\left[e^{s\Theta_d}\right]$ is given by
\[
\begin{array}{ccc}
\phi_{2k}(s)=\frac{e^{s\pi}-1}{s\pi}\prod_{j=1}^k \frac{(2j)^2}{(2j)^2+s^2} & &
\phi_{2k+1}(s)=\frac{e^{s\pi}+1}{2(s^2+1)}\prod_{j=1}^k \frac{(2j+1)^2}{(2j+1)^2+s^2}
\end{array}
\]
according to whether $d-2$ is even or odd respectively.

\item In particular $\EE[\Theta_d]=\frac{\pi}{2}$ for all $d$ and $\var[\Theta_d]=\beta_d$ where
\[
\beta_d:=\left\{
\begin{array}{lcl}
\dfrac{\pi^2}{4}-2\sum_{j=0}^k \dfrac{1}{(2j+1)^2}&&\text{ if $d-2=2k+1$ is odd or}\\
\dfrac{\pi^2}{12} -2\sum_{j=1}^k\dfrac{1}{(2j)^2}&&\text{ if $d-2=2k$ is even.}
\end{array}\right.
\]
\item The variance of the centered squared angle $\sigma_d^2 := \var \left(\Theta_d - \frac{\pi}{2}\right)^2$ is given by
\[\sigma_d^2 = \left\{
\begin{array}{ll}
 -\frac{\pi^4}{8} + 12 \sum\limits_{j=0}^{k} \frac{1}{(2j+1)^4} + 2\left( \frac{\pi^2}{4}-2\sum\limits_{j=0}^{k} \dfrac{1}{(2j+1)^2}\right)^2   & \text{ if $d-2=2k+1$ or}\\
 -\frac{\pi^4}{120} + 12 \sum\limits_{j=1}^{k} \frac{1}{(2j)^4} + 2 \left(\frac{\pi^2}{12} -2\sum\limits_{j=1}^{k}\dfrac{1}{(2j)^2}\right)^2  & \text{ if $d-2=2k$.}
\end{array} \right.\]

\end{enumerate}
\end{lemma}
\begin{proof} 
\begin{enumerate}
	\item Passing to polar coordinates $r,\phi_1,\dots, \phi_{d-1}$ with $r\leq 0$, $0\leq \phi_j\leq \pi$ for $1\leq j \leq d-2$ and $0\leq \phi_{d-1}\leq 2\pi$. The probability that $\Theta_d\leq \alpha$ is precisely the fraction of the surface area of the sphere defined by the inequality  $0\leq \phi_1\leq \alpha$. Since the surface element of the sphere is given by
\[ dS=\sin^{d-2}(\phi_1)\sin^{d-3}(\phi_2)\dots\sin(\phi_{d-2})d\phi_1\dots d\phi_{d-1}\]  
the probability is given by
\[
\frac{\int_0^{\alpha}\int_0^{\pi}\dots \int_0^{\pi} \int_0^{2\pi} dS}{\int_0^{\pi}\int_0^{\pi}\dots \int_0^{\pi} \int_0^{2\pi} dS}=\frac{\int_0^{\alpha}\sin^{d-2}(\phi)d\phi}{\int_0^{\pi} \sin^{d-2}\phi d\phi}
\]
as claimed. 
	\item We begin with a claim
	\begin{claim}
	Let $u: \RR \rightarrow \RR$ be a $C^2$-function and define $$\EE_d( u(x)) := \frac{\int_0^{\alpha}u(x)n\sin^{d-2}(\phi)d\phi}{A_d}$$ where $A_d := \int_0^{\pi} \sin^{d-2}\phi d\phi.$ Then, the following recursion formula holds 
	\[d \;\EE_d(u(x)) = {d}\;\EE_{d-2}(u(x)) - \dfrac{1}{d}\;\EE_{d}(u''(x)).\]
	\end{claim}
	\begin{proof}
	Using integration by parts one can show a recursive formula for $A_d$ and conclude that 
	\[A_d := \left\{\begin{array}{lcl}
	\frac{(2k)!}{(2^k \;d!)^2}\frac{\pi}{2} & & \text{if } d - 2 = 2k \text{ or} \\
	\frac{(2^k k!)^2}{(2k+1)!)}&& \text{if } d - 2 = 2k + 1. \\
	\end{array} \right.\]
	Applying integration by parts twice and using our formula for $A_d$ gives the result. 
	\end{proof}
	In particular if we take $u(x) = e^{sx}$ we get
	\[\EE_d(e^{sx}) = \left(\dfrac{d^2}{d^2 + s^2} \right)\EE_{d-2}(e^{sx}).\]
	As a result we obtain the stated closed formula for the moment generating function.
	\item All densities are, like sine, symmetric around $\frac{\pi}{2}$ and the first statement follows. To ease the computations we introduce the cumulant-generating function $ \psi_{d-2} = \log(\EE ( e^{s\Theta_d})).$ Then it is immediate that 
	$\Var(\Theta_d) = \psi_{d-2}''(0).$ 
	We consider two cases, $d$ even and odd, first let us assume that $d = 2k + 2$. Then, we write the cumulant-generating function as 
	\[\psi_{d-2}(s) = \underbrace{\log\left(\frac{e^{s\pi}-1}{s\pi}\right)}_{t(s)}+ \underbrace{\sum\limits_{j=1}^{k}\log\left(\dfrac{(2j)^2}{(2j)^2+s^2}\right)}_{r(s)}.\] 
	After some dry algebra we get $t''(0) = \frac{\pi^2}{12}$ and $r''(0) = 
	- 2\sum_{j=1}^{k} \frac{1}
	{(2j)^2}$, which gives the result for the even case. The odd case follows from an analogous argument.  
	\item Let $\mu_j$ be the $j$th moment of the random variable $\left(\Theta_d - \frac{\pi}{2}
	\right)$, i.e. $\mu_j = \EE\left(\Theta_d - \frac{\pi}{2}\right)^j$. It is well known that $\mu_2 = 
	\psi''(0)$ and $\mu_4 = \psi^{(4)}(0) + 3\left(\psi''(0)\right)^2.$ Therefore, 
\begin{equation} \label{eq:VarianceComulants}
	\Var \left(\Theta_d - \frac{\pi}{2}\right)^2 = \mu_4 - \mu_2^2 = \psi^{(4)}(0) + 2\left(\psi''(0)
	\right)^2.
\end{equation}
Again, consider two cases: $d$ even and $d$ odd. Suppose $d = 2k - 2$, just as before, we calculate 
$t^{(4)}(0) = -\frac{\pi^4}{120}$ and $r^{(4)}(0) = 12 \sum_{j=1}^{k} \frac{1}{(2j)^4}.$ Substituting 
both these into \eqref{eq:VarianceComulants} yields the claim. A similar argument can be applied to 
the odd case. 
  
\end{enumerate}
\end{proof}
At first glance the formulas for $\beta_d$ and $\sigma_d$ might seem a little complicated. In order to derive our results we need tangible decrease rates in terms of the dimension. The following claim gives us an easy way to interpret these quantities.

\begin{claim}\label{beta_sigma_bounds}
The following bounds hold for $\beta_d$ and $\sigma_d^2$:
\begin{itemize} 
	 	\item[]
 	\begin{equation}
 	\frac{1}{d} \leq \beta_d\leq \frac{1}{d-1} \qquad \qquad \text{for } d \geq 1,
 	\end{equation}
 	\item[] 
 	\begin{equation}
 	\dfrac{1}{2d^2}  \leq \sigma_d^2 \leq \dfrac{2}{(d-1)^2} \qquad \text{for }d \geq 4,
 	\end{equation}
 	moreover the upper bound for $\sigma_d^2$ holds for all $d\geq 1.$
 \end{itemize} 
\end{claim}
\begin{proof}
We distinguish two cases according on whether $d\geq 1$ is even or odd.
If $d$ is even, we can define $k$ by the equality $d-2=2k$ and compute 
\[\beta_d=2\sum_{j=k+1}^{\infty} \frac{1}{(2j)^2}.\]
Since this series consists of monotonically decreasing terms and $d\geq 2$ we conclude that
\[ \dfrac{1}{d} = \frac{1}{2k+2}=2 \int_{k+1}^{\infty} \frac{1}{(2x)^2}dx\leq 2\sum_{j=k+1}^{\infty} \frac{1}{(2j)^2}\leq 2\int_{k+\frac{1}{2}}^{\infty} \frac{1}{(2x)^2}dx=\frac{1}{2k+1} = \dfrac{1}{d-1}
\]
as claimed. On the other hand, notice that the other term concerning the variance can be written as
\[12 \sum\limits_{j=1}^{k}\dfrac{1}{(2j)^4} - \dfrac{\pi^4}{120} = - 12 \sum\limits_{j=k+1}^{\infty}\dfrac{1}{(2j)^4},\]
which again can be bound by 
\[\dfrac{2}{d^3} = \frac{1}{4(k+1)^3}=12 \int_{k+1}^{\infty} \frac{1}{(2x)^4}dx\leq 12
\sum_{j=k+1}^{\infty} \frac{1}{(2j)^4}\leq 2\int_{k+\frac{1}{2}}^{\infty} \frac{1}{(2x)^2}dx=\frac{2}
{(2k+1)^3} = \dfrac{2}{(d-1)^3}.\]
Then, we get 
\[\dfrac{1}{2d^2}  \leq \dfrac{2}{d^2} - \dfrac{2}{(d-1)^3} \leq \sigma_d^2 \leq \dfrac{2}{(d-1)^2} - 
\dfrac{2}{d^3} \leq \dfrac{2}{(d-1)^2}\]
where the first inequality follows since $d\geq 4.$
The case when $d$ is odd is  proven similarly.
\end{proof}

\section{Theoretical foundations}
\label{sec:3}

\subsection{Statement of results}\label{statements}

In this subsection we state the theoretical results that serve as basis for the proposed
methodology. Proofs are given in the following subsection. The setting is the following: An i.i.d.
sample, $X_1,\dots,X_n$, is available from a distribution $P$ on $\real^m$. Additionally we have access to a distingushied point $p$, and near this point the data live on a Riemannian $C^2$ manifold $M$, of dimension $d<m$. Furthermore, at $p$ the distribution $P$ has a Lipschitz continuous non-vanishing density function $g$, with respect to the volume measure on $M$. Without loss of generality,
we assume that $p = 0$. Then, we have

\begin{proposition}[\bf Behavior of nearest neighbors]\label{thm: Chernoff-Okamoto} 
For a positive constant $C$, define $k=\lceil C\log(n)\rceil$ and let $R(n)=L_{k+1}(0)$ be the euclidean distance in $\RR^m$ from $p=0$ to its $(k+1)$-st nearest neighbor in the sample $X_1,\dots, X_n$. Define $B_{R(n)}(0)$ to be the open ball 
of radius $R(n)$ around $0$ in $\real^m$. Then, the following holds true:
\begin{enumerate}
  \item For any sufficiently large $C> 0$, we have that, with probability one, for large enough $n$ ($n\geq n_0$, for some $n_0$ depending on the
actual sample), $R(n)\leq r(n)$, where
$$r(n):=O\left( \left(\frac{\log(n)}{n}\right)^{\frac{1}{d}}\right)$$
is a deterministic function that only depends on the distribution $P$ at $p$ and $C.$
\item Conditionally on the value of $R(n)$, the $k$-nearest-neighbors of $0$ in the sample $X_1,\dots, X_n$, 
have the same distribution as an independent sample of size $k$ from the 
distribution with density $g_n$,
equal to the normalized restriction of $g$ to $M\cap B_{R(n)}(0)$.
\end{enumerate}  
\end{proposition}

In what follows, with a slight abuse of notation, we will write $X_1,X_2,\dots, X_k$ to denote the
$k$ nearest neighbors of 0 in the sample and assume that these follow the distribution with
density $g_n$ of Proposition \ref{thm: Chernoff-Okamoto}. Let $\pi: \RR^m\rightarrow T_pM$ be the 
orthogonal projection onto the (embedded) tangent space to $M$ at $p=0$. For a nonzero
$X\in \real^m$, let $W:=\pi(X)$,
$\widehat{X}:=\frac{X}{\|X\|}$ and $\widehat{W}:=\frac{W}{\|W\|}$. $\widehat{W}$ takes values in the $(d-1)$-dimensional unit 
sphere $S^{d-1}$ of the tangent space of $M$ at $0$. 

Our first Lemma bounds the difference between the inner products 
$\langle\widehat{X_i},\widehat{X_j}\rangle$ 
and $\langle\widehat{W_i},\widehat{W_j}\rangle$ in terms of the length of projections. In this
Lemma, the random nature of the $X_i$ is irrelevant.

\begin{lemma}[\bf Basic projection distance bounds]\label{lem: basicIneqs} For any $X,X_1,X_2\in M$:
\begin{enumerate}
\item $\|X-\pi X\|=O(\|\pi X\|^2)$
\item $\|\widehat{X} - \widehat{W}\|=O(\|\pi X\|)$
\item The cosine of the angle between $X_1$ and $X_2$ is close to that between $W_1$ and $W_2$. 
More precisely, 
\[|\langle \widehat{X_1},\widehat{X_2}\rangle - \langle \widehat W_1,\widehat W_2\rangle| \leq Cr\]
for some $C\in \RR$, whenever $r\geq \|\pi(X_i)\|$ for $i=1,2$.
\end{enumerate}
\end{lemma}

Using Lemma \ref{lem: basicIneqs}, we can establish the following approximation. Let
$X_1,\dots, X_k$ be the $k$-nearest-neighbors from the sample to $p=0$ in $\RR^m$ . 
Define $W_i$ and $\widehat{W_i}$ as above and let $V_{k,n}$ be given by the formula
\[
V_{k,n}:=\frac{1}{\binom{k}{2}} \sum_{1\leq i<j\leq k} \left(\arccos\left\langle \widehat{W_i},\widehat{W_j}\right\rangle-\frac{\pi}{2}\right)^2.
\]

\begin{proposition}[\bf Approximating the statistic via its
tangent analogue]\label{thm: mainPart1} For $k=C\log(n)$, as above, we have
\begin{enumerate}
\item The sequence $k(U_{k,n}-V_{k,n})$ converges to $0$ in probability as $n\rightarrow \infty$.
\item $\lim_{n\rightarrow \infty}\EE\left(U_{k,n}-V_{k,n}\right)=0$.
\end{enumerate}
\end{proposition}

When $X$ comes from the distribution producing the sample, but is restricted to fall very
close to 0, the distribution of $\pi X$ will be nearly uniform in a ball centered at 0
in  $T_pM$. This will allow us to establish a coupling between the normalized projection
$\widehat{W}$ and a variable $Z$, uniformly distributed on the unit sphere of $T_pM$, an
approximation that leads to the asymptotic distribution of  $U_{k,n}$. 
Some geometric notation must be introduced to describe these results.
Since near 0, $M\subseteq \RR^m$ is a Riemannian submanifold of dimension $d$, 
it inherits, from the euclidean inner product in $\RR^m$, a smoothly varying inner product 
$l_p: T_pM\times T_pM\rightarrow \RR$, given by $l_p(u,v)=\langle i_{*}(u),i_{*}(v)\rangle$,
where $i: M\rightarrow \RR^m$ is the inclusion with differential $i_{*}$. This metric determines a 
differential $d$-form $\Omega_M$ which, in terms of local coordinates $\partial_i$ for $T_pM$ and dual 
coordinates $dx_i$ of $T_pM^*$ with $i=1,\dots d$, is given by $\Omega_M:=\sqrt{\det(\langle 
\partial_i,\partial_j\rangle)_{1\leq i,j\leq d})} dx_1\wedge\dots\wedge dx_d$. The differential form 
endows $M$ with a volume measure $\nu(U)=\int_U \Omega_M$. We say that a random variable $A$ on $M$ 
has density $g:M\rightarrow\RR$ if the distribution $\mu_A$ of $A$ satisfies $\mu_A(D)=\int_D g
\Omega_M$ for all borel sets $D$ in $M$.

If $X$ is a random variable taking values on $M$ with density $g$ and $r$ is a positive real number, 
let $X(r)$ be a random variable with distribution $g_r$ given by the normalized restriction of $g$ to 
$M\cap B_r(0)$, that is:

\[
g_r(z)=\begin{cases}
\frac{g(z)}{\int_{B_r(0)\cap M}g\Omega_M}\text{, if $z\in M\cap B_r(0)$ and}\\
0\text{, otherwise.}
\end{cases}
\]

Define $W(r):=\pi(X(r))$. The following geometric Lemma will be used 
for relating the densities of $X(r)$ and $W(r)$.

\begin{lemma}[\bf Tangent space approximations]\label{lem: basicGeom} The following statements hold for all sufficiently small $r$ and $p=0$ in $M$.
\begin{enumerate}
\item The map $\pi: B_{r}(0)\cap M\rightarrow \pi(B_{r}(0)\cap M)$ is a diffeomorphism. Let $\Phi: B_{r}(0)\cap T_pM \rightarrow B_{r}(0)\cap M$ be its inverse.  
\item The inclusion $\pi(B_{r}(0)\cap M)\subseteq B_r(0)\cap T_pM$ holds and moreover $|\lambda(B_r(0)\cap T_pM)-\lambda(\pi(B_{r}(0)\cap M))|=O(r)$ where $\lambda$ denotes the Lebesgue measure on $T_pM$.
\item The following equality holds:
\[ \left|1-\sqrt{\det\left\langle \frac{\partial \Phi}{\partial x_i},  \frac{\partial \Phi}{\partial x_j}\right\rangle_{1\leq i,j\leq d}}\right|=O(r)\]
\end{enumerate}
\end{lemma}

Let $D(r)$ be a random variable uniformly distributed in $B_r(0)\cap T_pM$ and note that 
$Z=\widehat{D}(r):=\frac{D(r)}{\|D(r)\|}$ is uniformly distributed on the unit sphere $S^{d-1}$, regardless of the value of $r$. Our next Lemma shows that under weak hypotheses there is a coupling between $W(r)$ and $D(r)$ which concentrates on the diagonal as $r$ decreases. 

\begin{lemma}[\bf Coupling]\label{lem: coupling} Let $r$ denote a small positive number. With $D(r)$ as above and $Z$ a random 
vector with the uniform distribution on the
unit sphere, $S^{d-1}$ of $T_pM$, 
if the density $g$ of $X$ in $M$, near 0, is locally Lipschitz continuous and  
nonvanishing at 0, then the following hold:
\begin{enumerate}
\item There exists a coupling $A(r)=(W(r),D(r))$ and a constant $C>0$ such that $\PP\{W(r)\neq D(r)\}
\leq Cr$ for all sufficiently small $r$. 
\item There exists a coupling $A'(r)=(\widehat{W(r)},Z)$ such that $\PP\left\{\widehat{W(r)}\neq Z
\right\}\leq Cr$ for all sufficiently small $r$. 
\end{enumerate}
\end{lemma}

The previous Lemma leads to the asymptotic distribution of the statistic $U_{k,n}$. 
\begin{theorem}[\bf Local Limit Theorem for angle-variance] \label{thm:main} Let $k:=\lceil C \log(n)
\rceil$ for the constant $C$ of the proof of Proposition \ref{thm: Chernoff-Okamoto} 
and assume $X_1,\dots, X_k$ are the $k$ nearest neighbors to $p=0$ in the sample, with respect
to the euclidean distance in $\RR^m$. If $\dim T_pM=d$ then the following statements hold:
\begin{enumerate}
\item The equality $\lim_{n\rightarrow \infty} \EE[U_{k,n}]=\beta_d$ holds and
\item The quantity $k\left(U_{k,n}-\beta_d\right)$ converges, in distribution, to that of 
$\>\sum_{i=1}^\infty\lambda_i(\chi_{1,i}^2-1)$ where the $\chi_{1,i}^2$ are i.i.d. chi-squared random 
variables with 
one degree of freedom and the $\lambda_i$ are the eigenvalues of the operator
$\calA$ on $L^2(S^{d-1})$ defined by
\[
({\calA}u)(x)=\int_{S^{d-1}}\left(h(x,z)-\beta_d\right)
u(z)\dds\mu(z)
\]
for $u\in L^2(S^{d-1})$, where $h(v,v'):=\left(\arccos(v\cdot v')-\frac{\pi}{2}\right)^2$ and $\mu$ denotes the uniform measure on $S^{d-1}$. 
\end{enumerate}
\end{theorem}
 This limit theorem is obtained by the various approximation steps given in the preliminary results
 together with the classical Central Limit Theorem for degenerate $U$ statistics, as described
 in Chapter 5 of \cite{serfling}. Depending on the relative values of the $\lambda_i$'s appearing in the statement of the Theorem, it could happen that the limiting distribution just obtained approaches a Gaussian distribution as the dimension increases
(this would happen if the $\lambda_i$ were such that Lindeberg's condition holds).

Although theoretical study of the $\lambda_i$'s is left for future work, we conjecture that as $d$ increases the limiting distribution converges to a Gaussian distribution. Numerical experiments seem to support our conjecture, see Figure~\ref{fig:qqplot2}.  

\begin{figure}[h]
\centering
\begin{subfigure}[b]{0.33\textwidth}
\includegraphics[width=\textwidth]{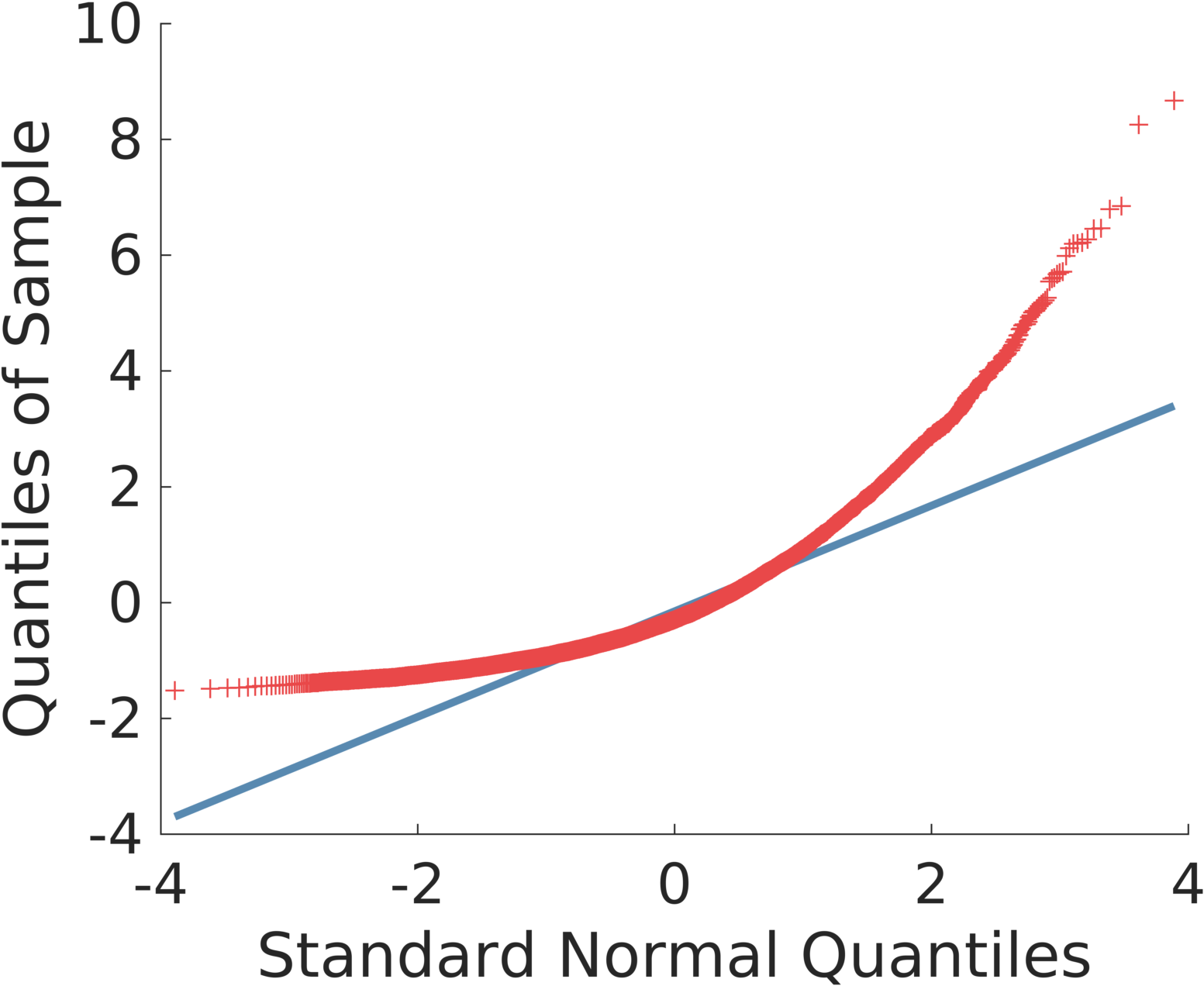}
\caption{$d = 2$}{}
\end{subfigure}
    ~ 
      \begin{subfigure}[b]{0.33\textwidth}
      \includegraphics[width=\textwidth]{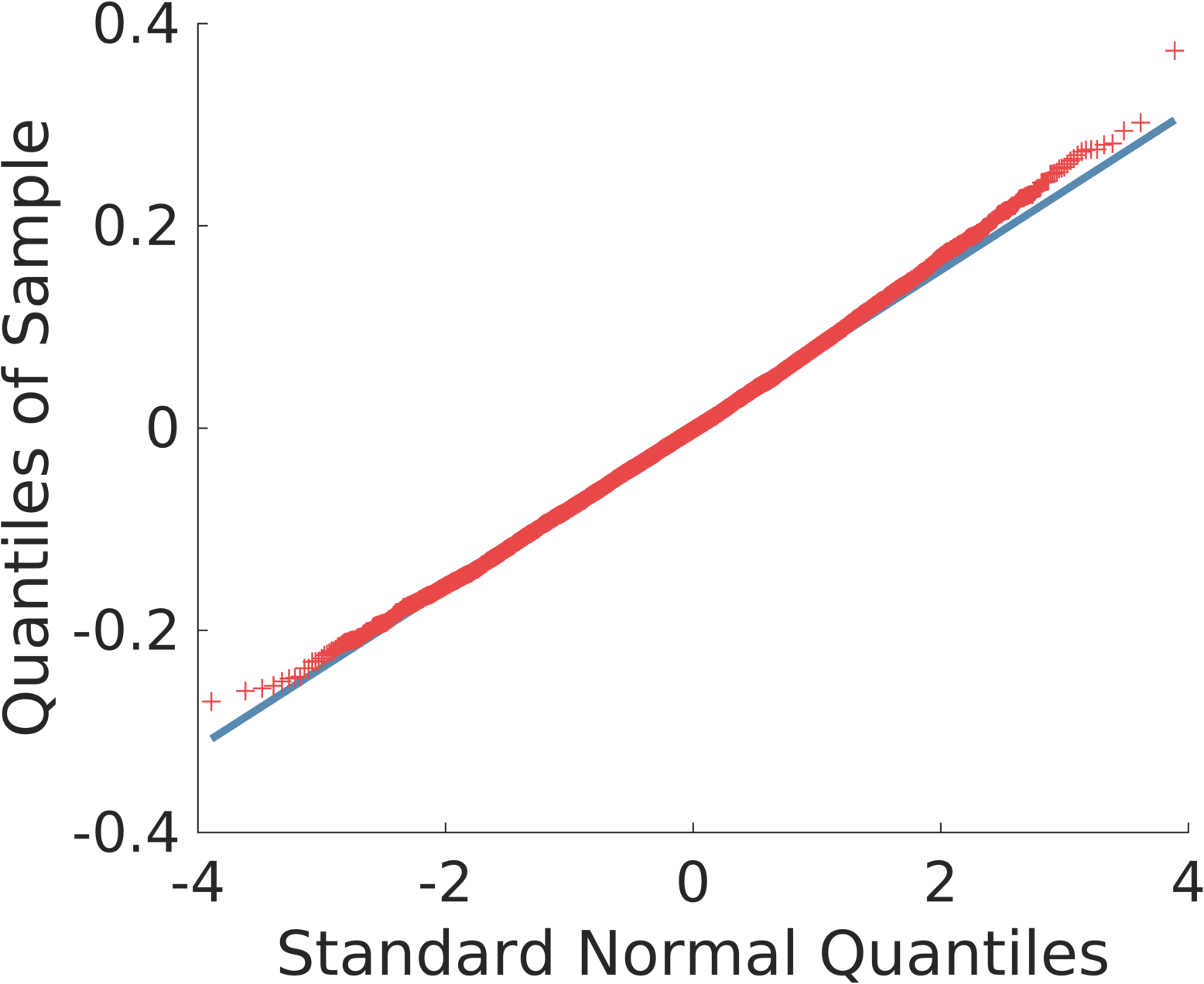}
      \caption{$d=25$}
      \end{subfigure}
      ~
      \begin{subfigure}[b]{0.33\textwidth}
      \includegraphics[width=\textwidth]{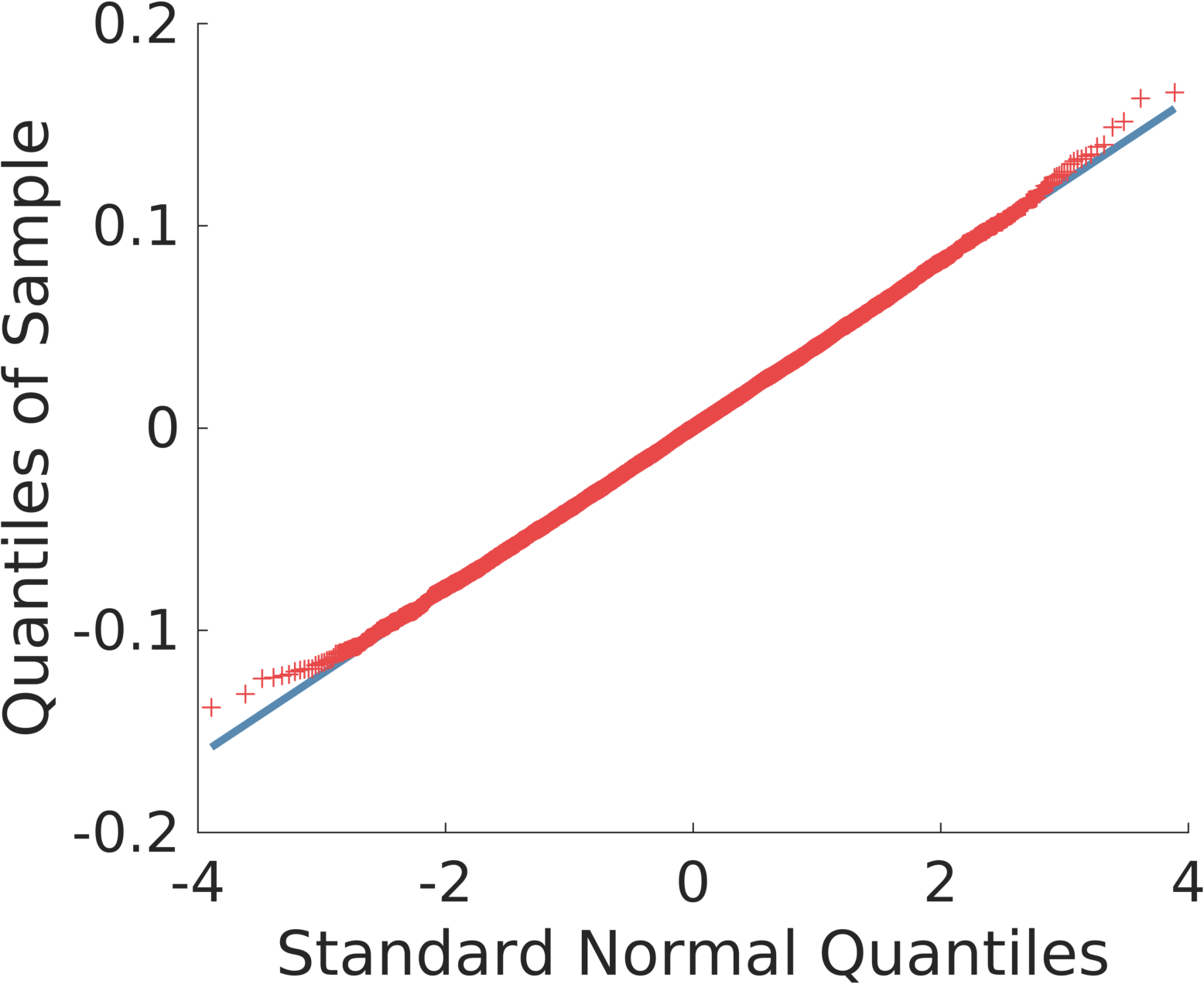}
      \caption{$d=50$}
      \end{subfigure}
      \caption{QQ-plots. As established in the proof of Theorem \ref{thm:main}, the limiting distribution is in fact the asymptotic distribution of $k(E_n-\beta_d)$, with $E_n$ defined in (Equation~\ref{E_n}). For this figure, we generate $10000$ samples of the variable $k(E_n-\beta_d)$, for $k=10d$ in dimensions $d=5, 25$, and $50$. The plots compare the quantiles of the sample distribution against the standard normal quantiles.}\label{fig:qqplot2}
      \end{figure}

In order to get a consistency result for our basic local dimension estimator, we will use the
following fact.

\begin{cor}[\bf Variance convergence]\label{cor: var_equal}
Under the conditions stated at the beginning of this subsection, 
\[
\lim_{n\to\infty}{k \choose 2}\var\, U_{k,n}=\sigma_d^2
\]
for $d$ equal to the dimension of $M$ near 0.
\end{cor}

Recall from Section~\ref{sec:2} that our basic procedure estimates the dimension $d$ as $\widehat d$, equal to the integer $r$ such that $\beta_{r}$ is closest to $U_{k,n}$. This procedure is consistent, as stated
next.

\begin{proposition}[\bf Consistency of basic dimension estimator]\label{prop: consist}
As in Section~\ref{sec:2}, write $\widehat{d}$ for the basic estimator described above. Let $d$ be the true
dimension of $M$ in a neighborhood of $p=0$. Then, in the setting of the present
section,
 $$\PP(\mbox{failure})=\PP(\widehat{d}\neq d)\rightarrow 0,$$ as $n\rightarrow \infty$.

\end{proposition}

\subsection{Proofs}{$\phantom 3_{\phantom 3}$}

\vskip 1pc
{\bf Proof of Proposition \ref{thm: Chernoff-Okamoto}}.
\neno Let us recall a probability bound for the Binomial distribution. For $N$ an integer
valued random variable with Binomial($n,p$) distribution and expected value
$\lambda=np$, one of the Chernoff-Okamoto inequalities (see Section 1 in \cite{Janson}) states
that, for $t>0$, $\PP(N\leq \lambda-t)\leq\exp(-t^2/2\lambda)$. Letting $t=\lambda/2$, we get
\beq\label{ch-ok1}
\PP(N\leq \lambda/2)\leq \exp(-\lambda/8). \enq
For fixed and small enough $r>0$, let $B_r(0)$ denote the ball of radius $r$ around 
$0\in \real^m$. For a random vector $X$, with
the distribution $P$ of our sample $X_1,X_2,\dots,X_n$, by our assumptions on $P$ and
$M$ near 0,
we have that 
\[
\PP(X\in B_r(0)\cap M)\leq\alpha\nu_d r^d,
\]
where $\nu_d$ is the volume (Lebesgue measure) of the unit ball in $\real^d$ and $\alpha$
is a positive number. 
Let $N=N_r$ denote the amount of sample points that fall in $B_r(0)\cap M$.
We have that $\lambda=\EE(N)\leq\alpha\nu_d r^d n$. We choose $r$ such that
$\lambda\leq C\ln n$, for a constant $C$ to be specified in a moment. Then,
by (\ref{ch-ok1}), we get
\beq\label{borel-c}
\PP\left(N\leq \frac{C}{2}\ln n\right)\leq\exp(-C\ln n/8)=\left(\frac{1}{n}\right)^{C/8}
\enq
Pick any value of $C>8$. For this choice, the bound in (\ref{borel-c}) will add
to a finite value when summed over $n$. By the Borel-Cantelli Lemma, the
inequality  $N> \frac{C}{2}\ln n$ will hold for all $n$ sufficiently large.
It follows that if $k=\frac{C}{2}\ln n$, the $k$-nearest-neighbors of 0 in the
sample, will fall in $B_r(0)$ for every $n$ sufficiently large and
the chosen value of $r$, namely
\[
r=r(n)=\left(\frac{C\ln n}{\alpha\nu_d n}\right)^{1/d}
\]
which is $\og((\ln n/n)^{1/d}))$. The proof of the first part of the
Proposition ends by renaming $C$.
\neno The statement of the second part of Proposition \ref{thm: Chernoff-Okamoto} is intuitive 
and has been used in the literature without proof. Luckily, Kaufmann and Reiss \cite{kr} 
provide a formal proof of these type of results in a very general setting. In particular, (ii)
of Proposition \ref{thm: Chernoff-Okamoto} holds by formula (6) of \cite{kr}.
\hfill$\Box$
\vskip 1pc
\noindent {\bf Proof of Lemma \ref{lem: basicIneqs}} \neno
By an orthogonal change of coordinates, we can assume that $T_pM$ is spanned by the first $d$ 
basis vectors in $\RR^m$. The projection $\pi: M\rightarrow T_pM$ is a 
differentiable function whose derivative at $p=0$ is the identity. 
By the implicit function theorem we can
conclude that there exists an $r>0$, such that 
$\pi: B_{r}(0)\cap M\rightarrow \pi(B_{r}(0)\cap M)$ is a 
diffeomorphism and that $M$ admits, near $p$ a chart
$\Phi:B_{r}(0)\cap T_pM \rightarrow M$ of the form
\beq\label{Phi_formula}
\Phi(z_1,\dots, z_d)=(z_1,\dots, z_d, F_1(z_1,\dots, z_d),\dots, 
F_{t}(z_1,\dots, z_d))
\enq
where $m=d+t$, $\Phi(0)=p=0$ and such that $\frac{\partial F_i}{\partial z_j}(0)=0$ for 
$1\leq i \leq t$ and $1\leq j\leq d$.
As a result, the euclidean distance between a point of $M$ near $0$ and the tangent space at 
$0$ is given, in the local coordinates $z$, by 
\[d(z_1,\dots, z_d)=\sqrt{\sum_{i=1}^t F_i^2(z)}.\]
We will prove that there exists a constant $K$ such that, for all sufficiently small $\delta 
>0$ and all $z$ with $\|z\|\leq \delta$ the inequality
$d(z)\leq K \|z\|^2$ holds.
By Applying Taylor's Theorem at $0$ to the differentiable function $d(z)$ we conclude, since 
$\Phi(0)=0$ and $\frac{\partial F_i}{\partial z_j}(0)=0$, that the constant and linear term 
vanish from the expansion. This proves the claim because $\|z\|^2=\|\pi(\Phi(z))\|^2$. 
Assume $K$ is a constant which satisfies  
$\|X-\pi X\| \leq K \|\pi X\|^2$. Thus 
\begin{equation} \label{eq:dist1} \left\|\frac{X}{\|\pi X\|} - \frac{\pi X}{\|\pi X\|}\right\| \leq K \|\pi X\|.\end{equation}
On the other side, 
\begin{align}\label{eq:dist2}\begin{split}
\left\|\frac{X}{\|X\|} -\frac{X}{\|\pi X\|}\right\| &= \|X\|\left|\dfrac{1}{\|X\|} - \dfrac{1}{\|\pi X\|} \right| \\ 
& = \dfrac{1}{\|\pi X\|}\left|\|X\| - \|\pi X\|\right| \leq \dfrac{1}{\|\pi X\|}\|X - \pi X\| \leq K \| \pi X \|.
\end{split}
\end{align}	
Altogether, 
\[
\left\|\frac{X}{\|X\|} - \frac{\pi X}{\|\pi X\|}\right\| \leq \left\|\frac{X}{\|X\|}-\frac{X}{\|\pi X\|}\right\| + \left\|\frac{X}{\|\pi X\|} - \frac{\pi X}{\|\pi X\|} \right\| \leq 2 K \|\pi X\|,
\]
where the first inequality is just the triangle inequality and the second one follows from \eqref{eq:dist1} and \eqref{eq:dist2}. 
The third item in the Lemma follows immediately from the triangle inequality and the second item by adding and subtracting $\langle \widehat{X_i}, \widehat W_j\rangle $.
\hfill$\Box$
\vskip 1pc
\noindent 
\begin{remark}
The quadratic term of $G$ is the second fundamental form of $M$ and, therefore, the 
constant $K$ can be chosen to be the largest sectional curvature of $M$ at $p$. 
\end{remark}

Before proving Proposition \ref{thm: mainPart1}, we need a Lemma on the behavior of
the $\arccos$ function.
\begin{lemma}\label{lem: basicArccos} Suppose that $-1\leq c_1\leq c_2\leq 1$ and let 
$\delta=c_2-c_1$ be sufficiently small (for our purposes it suffices to have
$\delta\leq 1/4$). Then, 
\[
|\arccos(c_2)-\arccos(c_1)|\leq 2\sqrt{|c_2-c_1|}
\]
\end{lemma} 
\begin{proof} Assume first that both $c_1$ and $c_2$ are positive. We have
\[
|\arccos(c_2)-\arccos(c_1)| = \int_{c_1}^{c_2} \frac{1}{\sqrt{1-x^2}}dx.
\]
Using that the integrand in the last expression is increasing in $[0,1]$ and by the change
of variables $u=1-x$, we get
\[
|\arccos(c_2)-\arccos(c_1)|\leq \int_{1-\delta}^{1} \frac{1}{\sqrt{1-x^2}}dx=
\int_{0}^{\delta} \frac{1}{\sqrt{u(2-u)}}du\leq \int_{0}^{\delta} \frac{1}{\sqrt{u}}du
\]
since $2-u\geq 1$ for $u\leq\delta$. From this last bound, the result
follows in this case by integration. The argument for the case in which both $c_1$ and $c_2$ 
are negative is identical, by symmetry. In the case $c_1\leq 0 \leq c_2$, both $c_1$ and $c_2$ 
fall in a fixed interval ($[-1/4,1/4]$) where the derivative of $\arccos$ is bounded and the result follows easily. 
\end{proof}

\noindent {\bf Proof of Proposition \ref{thm: mainPart1}}
\neno To prove part (1), putting together Proposition~\ref{thm: Chernoff-Okamoto} and 
Lemma~\ref{lem: basicIneqs} we have  
\[
\max_{i\leq k}\|\widehat{X_i} - \widehat{W_i}\|=\og(r(n))=\og\left(\left(\frac{\ln n}{n}
\right)^{1/d}\right)
\]
From this, it follows easily that 
\[
\max_{i< j\leq k}
|\langle\widehat{X_i},\widehat{X_j}\rangle - 
\langle\widehat{W_i},\widehat{W_j}\rangle|=\og(r(n)),
\]
and, using Lemma~\ref{lem: basicArccos} we get 
\[
\max_{i< j\leq k}
|\arccos\langle\widehat{X_i},\widehat{X_j}\rangle - 
\arccos\langle\widehat{W_i},\widehat{W_j}\rangle|=\og(\sqrt{r(n)})
\]
The bound is preserved by the application of the function $u\mapsto (u-\pi/2)^2$ (since the function is locally Lipschitz) and
by taking averages over all pairs, and we get
\beq\label{UV_bound}
U_{k,n}-V_{k,n}=\og(\sqrt{r(n)})=\og\left(\left(\frac{\ln n}{n}\right)^{1/2d}\right).
\enq
The result follows by observing that, for the value of $k$ considered,
\[
k\,\og\left(\left(\frac{\ln n}{n}\right)^{1/2d}\right)=\och(1).
\]
To prove $(2)$, notice that from part $(1)$ it is immediate that $|U_{k,n}-V_{k,n}|$ converges 
to zero in probability, which implies that $\lim_{n\rightarrow \infty}\EE\left(U_{k,n}-V_{k,n}
\right)=0$, since $U_{k,n}-V_{k,n}$ is a bounded random variable. 
\hfill$\Box$
\vskip 1pc
\noindent {\bf Proof of Lemma \ref{lem: basicGeom}}
\neno
Recall, from the proof of Lemma \ref{lem: basicIneqs}, that for $r>0$, small enough, the projection
$\pi: B_{r}(0)\cap M\rightarrow \pi(B_{r}(0)\cap M)$ is a 
diffeomorphism and that $M$ admits, near $p$ a chart (inverse)
$\Phi:B_{r}(0)\cap T_pM \rightarrow M$ of the form
given in (\ref{Phi_formula}) and satisfying that $\Phi(0)=p=0$ 
and such that $\frac{\partial F_i}{\partial z_j}(0)=0$ for 
$1\leq i \leq t$ and $1\leq j\leq d$.
Also from that proof, recall that there exists a constant $K$ such that for  
small enough $\delta >0$ and all $z$ with $\|z\|\leq \delta$, we have
$d(z)\leq K \|z\|^2$, where $d(z)$ is the distance between a point $z\in M$
and its projection on $T_pM$. It follows that the image $\pi(B_r(0)\cap M)$ contains a ball 
of radius $r'<r$ such that $r-r'=O(r^2)$ and therefore
\[B_r(0)\cap T_pM\supseteq \pi(B_r(0)\cap M))\supseteq   B_{r'}(0)\cap T_pM,\]
proving part $(2)$ of the Lemma, since the volume of the first and last term differ 
by at most $O(r^2)$. For part $(3)$ 
note that $\frac{\partial \Phi}{\partial z_i}=e_i + 
\sum_{t=1}^{m-d} \frac{\partial F_t }{\partial z_i} e_{d+t}$. Since $\frac{\partial F_t }{\partial z_i}$ 
are $O(r^2)$ the inner products $\langle \frac{\partial \Phi}{\partial z_i}, \frac{\partial \Phi}
{\partial z_j}\rangle $ are $1+O(r^2)$ if $i=j$ and $O(r^2)$ otherwise and we conclude that $
\sqrt{\langle \frac{\partial \Phi}{\partial z_i}, \frac{\partial \Phi}{\partial z_j}\rangle}$ is 
$1+O(r)$ as claimed.
\hfill$\Box$
\vskip 1pc
\noindent {\bf Proof of Lemma \ref{lem: coupling}}
\neno The total variation distance between two probability measures $\mu$ and $\nu$, 
defined as $\|\mu-\nu\|_{TV}:=\sup_{A}|\mu(A)-\nu(A)|$, satisfies
  \[\|\mu-\nu\|_{TV} =\inf\{\PP\{X\neq Y\}: A=(X,Y)\}\]
  where the infimum runs over all couplings $A=(X,Y)$ of random variables $X,Y$ with distributions 
  given by $\mu$ and $\nu$, respectively (see, for instance, page 22, Chapter 1 of \cite{villani}). Moreover, if $\mu$ and $\nu$ are given by densities $g_1,g_2$ 
  then the following inequality holds
  \[\|\mu-\nu\|_{TV}\leq \|g_1-g_2\|_{L^1}.\]
  
  We will prove the first part of the 
  Lemma by bounding the $L^1$-norm of the difference of the densities of $W(r)$ and 
  $D(r)$. Recall the definition of $W(r)$ right before the statement of Lemma \ref{lem: basicGeom}.
  The difference of the two densities  is given by
  \[
  \int_{B_r(0)\cap T_pM}\left| h_r-\frac{1}{\lambda(B_r)}\right|d\lambda
  \]
  where $h_r$ denotes the density of $W(r)$ with respect to the $d$-dimensional Lebesgue measure 
  $\lambda$ in $T_pM$ and $B_r:= B_r(0)\cap T_pM$. More precisely, defining $\Phi$ as in 
  Lemma~\ref{lem: basicGeom} the density of $W(r)$ is given by
  \[h_r(u)= g_r(\Phi(u))\sqrt{\det\left\langle \frac{\partial \Phi}{\partial x_i}, \frac{\partial \Phi}{\partial x_j}\right\rangle_{1\leq i,j\leq d}(u)}\]

  Since $g$ is locally Lipschitz continuous there exists a constant $K$ such that 
  \[ g(0)-K_1r\leq g(\Phi(u))\leq g(0)+K_1r.\]
  By Lemma~\ref{lem: basicGeom} there exist constants $K_2,K_3$ such that the following inequalities hold for $u\in B_r$:
  \[
  1-K_2r\leq  \sqrt{\det\left\langle \frac{\partial \Phi}{\partial x_i},  \frac{\partial \Phi}{\partial x_j}\right\rangle_{1\leq i,j\leq d}(u)}\leq 1+K_2r \text{ and}
  \]
  \[\lambda(B_r)-K_3r\leq  \lambda(\pi(B_r(0)\cap M))\leq \lambda(B_r)+K_3r\]
  Combining these inequalities we conclude that there exists a constant $\tilde{K}$ such that for all $u\in B_r$ 
  \[ \lambda(B_r)(g(0)-\tilde{K}r)\leq \int_{B_r(0)\cap M} g\Omega \leq \lambda(B_r)(g(0)+\tilde{K}r).\]
  As a result the inequality
  \[ \frac{1}{\lambda(B_r)}\left(\frac{g(0)-\tilde{K}r}{g(0)+\tilde{K}r}-1\right) \leq h_r-\frac{1}{\lambda(B_r)}\leq \frac{1}{\lambda(B_r)}\left(\frac{g(0)+\tilde{K}r}{g(0)-\tilde{K}r}-1\right)\]
  so, using the fact that $g(0)>0$ we conclude that there exists a constant such that
  \[\int_{B_r} \left|h_r-\frac{1}{\lambda(B_r)}\right|\leq Cr\] as claimed, finishing the
  proof of the first part of the Lemma. \neno
  For part (2), let $A'(r)$ be the random pair obtained from $A(r)=(W(r),D(r))$ by normalizing its 
  components, that is  $A'(r)=\left(\widehat{W(r)}, \widehat{D(r)}\right)$ and note that $\PP\left\{\frac{W(r)}{\|W(r)\|}\neq \frac{D(r)}{\|D(r)\|}\right\}=O(r)$, because this probability is bounded above by the probability that $W(r)$ and $D(r)$ differ. Note that the random vector $Z=\frac{D(r)}{\|D(r)\|}$ is uniform on the unit sphere, and in particular its distribution is independent of the value of $r$.
\hfill$\Box$
\vskip 1pc
\noindent For $k=\lceil C\log(n)\rceil$, as before,
 let $Z_1,\dots, Z_k$ be an i.i.d. sample distributed uniformly on the unit sphere $S^{d-1}$ of $T_pM$ and define
\beq\label{E_n}
E_n:=\frac{1}{\binom{k}{2}}\sum_{1\leq i<j\leq n} \left(\arccos\langle Z_i,Z_j\rangle-\frac{\pi}{2}\right)^2.\enq
In view of Proposition \ref{thm: mainPart1}, in order to prove Theorem \ref{thm:main} it 
will suffice to show that the limiting standardized distribution of the $V_{k,n}$ defined in that Proposition 
is the same as that of $E_n$, and establish the asymptotics for $E_n$.

\vskip 1pc
\noindent {\bf Proof of Theorem \ref{thm:main}}

By Proposition~\ref{thm: Chernoff-Okamoto} part $(1)$ on a set of probability 1 (on the set
of infinite samples of data), for some $n_0$, the inequality $R_n:= R(n) \leq r(n)$ holds for
$n\geq n_0$.

\noindent
By Proposition~\ref{thm: Chernoff-Okamoto} part $(2)$, for $i=1,\dots, k$ the distribution of $W_i=\pi(X_i)$, the projections on $T_pM$ of the $k$ nearest neighbors of $0$ in the sample, is that of
an independent sample $W_1(R_n),\dots, W_k(R_n)$, with the $W_j(R_n),\,j\leq k$ as defined before
Lemma \ref{lem: basicGeom}. 

From the previous Lemma, conditionally on $R_n$, we have a coupling 
$A_j'(R_n)=(\widehat{W_j(R_n)},Z_j)$ for each $j\leq k$. These couplings can be taken such that
the $\widehat{W_j(R_n)}, j\leq k$ form an i.i.d. sample and the same holds for the $Z_j$'s. By
the previous lemma, we have that for each $j$, $\PP(\widehat{W_j(R_n)}\neq Z_j)\leq C r(n)$.
Then, except for a set of measure 0, we get the following event inclusion
\[
\{kV_{k,n}\neq kE_n\}\subseteq \bigcup_{j=1}^k\{ \widehat{W_i(R_n)}\neq Z_i\}
\]
By the union bound, the probability of the rightmost event is bounded by $Ck(n)r(n)$ which
goes to zero, as $n$ goes to infinity by the choice of $k(n)$ and the value of
$r(n)$.  It follows that $V_{k,n}$ and $E_n$ have the same standardized asymptotic distributions, and being both random variables bounded, it follows that $\lim_{n\rightarrow \infty}\EE(V_{k,n}-E_n)=0$.

For part (2) of Theorem \ref{thm:main} it only remains to establish the limiting 
distribution of $E_n$. This statistic falls in the
framework of classical $U$-statistics, that have played an important role in the theory
of many non-parametric procedures. See for instance \cite{rw79} for several applications of
the theory of $U$-statistics. Even an empirical processes theory is available for $U$-processes,
see for instance \cite{ag93}, which has found application in the study of notions of multivariate
depth. Still, we will only require the classical theory, as exposed in Chapter 5 of \cite{serfling}
and Chapter 3 of \cite{rw79}.

Recall the definition of the kernel $h$ in Equation~\ref{kernel}. By the symmetry of the uniform distribution for three independent vectors,
$Z_1,Z_2,Z_3$ with uniform distribution on $S^{d-1}$, it can
be easily verified that, 
\beq\label{dege}
\esp\left(h(Z_1,Z_2)h(Z_1,Z_3)-\beta^2_d\right)=0.
\enq 

This means that the $U$-statistic associated to $h$ is degenerate. It follows (see the variance
calculation in \cite{rw79}) that
\beq\label{var_asymp}
{k\choose 2}\mbox{Var}(E_n)=\mbox{Var}(h(Z_1,Z_2))
\enq
and by the Theorem for degenerate $U$-statistics in Section 5.5 of \cite{serfling}, part (2)
of our Theorem follows.
\hfill$\Box$
\vskip 1pc
\noindent 

\vskip 1.5pc
\noindent {\bf Proof of Corollary \ref{cor: var_equal}}

\noindent By (\ref{var_asymp}), it suffices to show that 
$k^2\lim(\EE U_{k,n}^2-\EE E_n^2)=0$, for $E_n$ as in 
the proof of Theorem \ref{thm:main}. Applying equation (\ref{UV_bound}) and the
facts that the function $h$ is bounded and that $k=\mbox{O}(\ln n)$ on a set
of probability 1, we have
\[
k^2\EE U_{k,n}^2\leq k^2\EE V_{k,n}^2+\mbox{O}\left(\ln^2 n
\left(\frac{\ln n}{n}\right)^{1/2d}\right).
\]
The opposite inequality (interchanging the roles of $U_{k,n}$ and $V_{k,n}$) is obtained by
the same reasoning, and we get 
\beq\label{varlim1}
\lim k^2(\EE U_{k,n}^2-\EE V_{k,n}^2)=0.
\enq
By the coupling argument of the proof of Theorem \ref{thm:main}, we have that $V_{k,n}$ and
$E_n$ might differ at most on a set of measure $\mbox{O}(\ln (n)\, r(n))$. Taking expectations
on the sets were they coincide and differ, we obtain
\[
k^2\,\EE V_{k,n}^2\leq k^2\,\EE E_{n}^2+k^2\,\mbox{O}(\ln (n)\, r(n)).
\]
Observing that the second term in the right hand side of this
inequality goes to zero, as $n$ grows, and that the reverse
inequality is obtained similarly, we conclude
\beq\label{varlim2}
\lim k^2(\EE V_{k,n}^2-\EE E_{n}^2)=0,
\enq
and the result follows by combining (\ref{varlim1}) and (\ref{varlim2}).
\hfill$\Box$
\vskip 1pc
\noindent {\bf Proof of Proposition \ref{prop: consist}}.

\noindent Let $d$ be the true value of the dimension. 
$\beta_{d+1}$ is the expected value closest to $\beta_d$ and, by the proof of Claim 
\ref{beta_sigma_bounds}, $(\beta_d-\beta_{d+1})/2\geq 1/(d-1)^2$. For the basic
procedure to incur in error it is necessary that 
$|U_{k,n}-\beta_d|\geq |U_{k,n}-\beta_{d+1}|$, which means
\[
|\EE U_{k,n}-\beta_d|+ |U_{k,n}-\EE U_{k,n}|\geq \frac{\beta_d-\beta_{d+1}}{2}\geq
\frac{1}{(d-1)^2}.
\]
Since $|\EE U_{k,n}-\beta_d|$ converges to zero, the condition above requires that 
$|U_{k,n}-\EE U_{k,n}|\geq 1/\frac{1}{2(d-1)^2}$, for $n$ large enough.
The probability of this last event is bounded as follows. Let $c$ be a sup norm bound 
for $h(\widehat{X_1},\widehat{X_2})-\EE h(\widehat{X_1},\widehat{X_2})$, for
the kernel $h$ given
in (\ref{kernel}) and $\widehat{X_i}$ as in Lemma \ref{lem: basicIneqs}. Clearly,
$c$ is bounded above by $\pi^2/4$. By Bernstein's inequality for
$U$-statistics \cite[Proposition 2.3(a)]{ag93}, we get
\[
\PP\left(|U_{k,n}-\EE U_{k,n}|\geq \frac{1}{2(d-1)^2}\right)\leq
2\exp\left(\frac{-k/(8(d-1)^4)}{2\mbox{Var}(h(\widehat{X_1},\widehat{X_2}))
+\frac{2c}{6(d-1)^2}}\right)
\]
Now, using Claim 
\ref{beta_sigma_bounds} and Corollary \ref{cor: var_equal}, we have, after some calculations,
\beq\label{cota_error}\begin{small}
\PP\left(|U_{k,n}-\EE U_{k,n}|\geq \frac{1}{2(d-1)^2}\right)\leq
2\exp\left(\frac{-k/(8(d-1)^4)}{\frac{5}{(d-1)^2}+\frac{\pi^2}{12(d-1)^2}}\right)
\leq 2\exp\left(\frac{-k}{47(d-1)^2} \right)
\end{small}
\enq
This bound goes to zero as $n$ (and $k$) grow to infinity, finishing the proof.
\hfill$\Box$
\vskip 1pc
\noindent Forcing the estimation error bound in (\ref{cota_error}) to be less that a given
$\delta>0$ will give a value of $k$ of the order of $Ad^{\,2}\ln(2/\delta)$, for some constant
$A$, reflecting that precise estimation is more demanding, in terms of sample size,
as the dimension $d$ grows.

\section{Estimators}\label{sec:4}
\label{sec: Applications}

In this section we present two dimension estimators based on the statistic $U_{k,n}$. First,
 a local estimator that gives the dimension of $M$ around a distinguished non-singular point 
 $p \in M$ is discussed. Then, the case in which the manifold is equidimensional is considered, by 
building upon our local estimator to propose a global dimension estimator. Some implementation 
issues are discussed and in the following section we evaluate the performance of our estimators. The programming code used in these experiments is publicly available, 
it can be found at \url{https://github.com/mateodd25/ANOVA_dimension_estimator}.
\begin{Ualgorithm}[h]
\SetAlgoLined
\KwData{$k \in \NN_+$ and  $X_1, \dots, X_n, p \in M \subseteq \RR^m$}
\KwResult{Estimated dimension $\widehat d$ at $p \in M$}
Find the $k$-nearest neighbors to $p$\;
Use these neighbors to compute $U_{k,n}$ as in \eqref{def:statistic}\;
Choose ${\widehat d}$ associated with $U_{k,n}$\;
\caption{Local dimension estimation}\label{alg:local_estimator}
\end{Ualgorithm} 

\subsection{Local estimators}

The theory presented in Section~\ref{sec:3} suggests that $k \sim \log(n)$ should be a good choice, asymptotically-speaking, for the number of neighbors to consider in the local dimension
estimation procedure. However, one could potentially leverage prior knowledge of the structure of the problem to set this parameter differently. In our implementation we set it to $k = \text{{\tt round}} (10 \log_{10}(n))$. 

In our theoretical analysis presented above, it was assumed that we are given a center point $p$
where the local dimension is to be estimated. A natural question that arises in practice
is the following: given a sample, how to select good center points. In Section~\ref{subsec:global_estimators} we present a simple heuristic to select ``good'' centers. 

Both estimators presented in what follows are based on the Algorithm~\ref{alg:local_estimator}. The 
difference between the estimators considered lies on the last line of the algorithm, namely, on how
to pick the dimension estimator, given $U_{k,n}$. Next, the two different rules to execute this 
step are discussed. 


\subsubsection{Basic estimator} Since $U_{k,n}$ converges in probability to
$\beta_d$, a natural way to estimate the dimension from $U_{k,n}$ is to set 
\[\widehat d_{\text{basic}} := \argmin_{d \in [D_{\rm max}]} |\beta_d - U_{k,n}|,\]
where $D_{\rm max}$ is the ambient dimension or some bound we know a priori on the dimension of the manifold.
Interestingly, such a rule is fairly accurate, as established in Proposition \ref{prop: consist}. Another advantage of this estimator is that there is no need to train it, since all the quantities involved have been analytically computed (and presented in Section\ref{sec:2}). 

\begin{remark}
Classical discriminant analysis results (see, for example, Section 4.1 in \cite{Devroye2013})
would advice to incorporate available variance information on the selection of $\widehat{d}$,  by
choosing
\[\widehat d_\mathrm{disc} := \argmax\{d \mid U_{k,n} \geq \eta_d\}\qquad \text{where} \qquad \eta_d = \beta_d + \dfrac{\sigma_{d}}{\sigma_{d} + \sigma_{d+1}}(\beta_{d-1} - \beta_d).\]
Still, simulation evaluations
(not included) show that $\widehat d_{\mathrm{disc}}$ and 
$\widehat d_\mathrm{basic}$ have a very similar performance in practice (and also in theory, since
both are consistent). Thus, we prefer to use the later, being the simpler one.
\end{remark}

\subsubsection{Kernel-based estimator}

For our second estimator we start by simulating multiple instances $Y_1^{(d)}, \dots, Y_M^{(d)}$ of 
the random variable $k(E_n-\beta_d)$, for a large value of $M$ ($=$5000, for instance) and
with $E_n$ as defined in equation~\eqref{E_n}, for each 
dimension $d \in [D_{\rm max}]=\{1, \dots, D_{\rm max}\}$.  From these data, the density $\hat f^{(d)}_k$, of
$k(E_n-\beta_d)$, is estimated, for each $d$, as 
\beq\label{kerneldens}
\hat f^{(d)}_k(y)=\frac{1}{Mh}\sum_{i=1}^M\varphi\left(\frac{y-Y_i^{(d)}}{h}\right)
\enq 
where $\varphi(\cdot)$ is the standard Gaussian density and the parameter $h$ (the ``bandwidth'')
can be set at $h=(4/3M)^{1/5}$.
This choice of bandwidth guarantees consistent density estimation (see 
Section 4.1 in \cite{bf}).
Then, a Bayesian classification procedure with uniform prior distribution on the
set $[D_{\rm max}]$, would select the dimension as that for which
the kernel density estimator is maximized at the standardized $U_{k,n}$, namely
\[\widehat d_{\rm ker} = \argmax_{d \in [D_{\rm max}]} \hat f_k^{(d)}(k(U_{k,n}-\beta_d)).\] 
Notice that the simulations described above need to be performed only once for each dimension,
since they are made on uniform data on $S^{d-1}$ and do not depend on the particular
data being studied. 

\subsection{Global estimators}\label{subsec:global_estimators}

We now turn our attention to extending the local dimension estimation algorithm to a global one. Assuming that $M$ is equidimensional, the local method can be extended by running multiple instances of Algorithm~\ref{alg:local_estimator} on different centers, and combining the results,
as outlined in Algorithm~\ref{alg:global_estimator}. 

After getting dimension estimates at each center, one could use different 
summary statistics to choose the global dimension, such as the mean, the mode or the median. To make a 
method robust against outliers, we chose to use the median. To decide about the parameter $c$ we ran a 
cross-validation algorithm. Empirically, it appears that $c \sim \log(n)$ is a good choice for this 
parameter. 

\begin{Ualgorithm}[t]
\SetAlgoLined
\KwData{$c \in \NN_+$, $k \in \NN_+$ and  $X_1, \dots, X_n \in M \subseteq \RR^m$}
\KwResult{Estimated dimension $\widehat d$ of $M$}
Choose $c$ centers $p_1, \dots, p_c$ from the sample\;
Apply Algorithm~\ref{alg:local_estimator} to each center $p_i$, let $\widehat d_i$ be its output\;
Set ${\widehat d}$ to the median of $\{\widehat d_i\}_{i=1}^c$\;
\caption{Global dimension estimation}\label{alg:global_estimator}
\end{Ualgorithm}

\subsubsection{Choosing centers} To pick the centers $p_i$ in Algorithm~\ref{alg:global_estimator},
we divide the sample into $c$ disjoint subsamples of approximately equal size. Assume, for simplicity of the exposition, that $c = 1$. Inside each subsample we pick a center by assigning each point a score of centrality and then choosing the one with highest score. Scores 
are assigned through the following procedure: 
\begin{enumerate}
	\item For each coordinate $i$, we order the sample based on the $i$-th entry, let 
	$\tau_i$ be the permutation giving this ordering, that is, 
	the $i$-th row in $(X_{\tau_i(1)}, \dots, 
	X_{\tau_i(n)})$ is nondecreasing.
	\item Then, the centrality score of $X_j$ is given by $\sum_{i=1}^m f(\tau_i(j))$, where $f(x) = \left|\frac{1}{2} - \frac{2(x-1)}{2n}\right|.$
\end{enumerate} 
For the $i$-th coordinate, the weight function $f$ gives the maximum scores to 
the point (or points) such that $\tau_i(j)$ is closest to $n/2$ and thus,
the mechanism used chooses as center a point which for many components, appears 
near the center of these orderings.

\subsubsection{Heuristic to discard centers} Finally, we present a simple heuristic for discarding some of the selected centers, based on the mean of the angles between its neighbors, taking as
always, the point considered as origin. This is done in order to improve the performance of the
statistic. Consider again the angle
\[
\theta_{i,j}=\arccos\left\langle \frac{X_i-p}{\|X_i-p\|},\frac{X_j-p}{\|X_j-p\|}\right\rangle
\]
for each pair of nearest neighbors $X_i,X_j$ of the point $p\in M$, as used in the basic
definition (\ref{def:statistic}).  
Consider the average of these angles, \[\overline{\theta}(p)=
\frac{1}{\binom{k}{2}} \sum_{1\leq i<j\leq k} \theta_{i,j}.
\]
If the manifold $M$, near $p$, is approximately flat, $\overline{\theta}(p)$ should be close
to $\pi/2$, regardless of the value of the dimension $d$, since $\pi/2$ is the expected
value of the angle between uniformly sampled points in every dimension and the $U$-statistic
$\overline{\theta}(p)$ should converge rapidly to this expectation.
Thus, when $\overline{\theta}(p)$ is far from $\pi/2$, it can be taken as a suggestion of 
strong curvature that is causing non-uniformity of the angles, and therefore, $p$ might
not be a good point to consider for dimension estimation. For these reason, in our 
implementation, the user is allowed to use this heuristic and discard a fraction
of the centers $p_i$ with largest values of $|\overline{\theta}(p_i)-\pi/2|$. Experiments presented in the next section suggest that this heuristic is 
useful when the manifold is highly curved.  

\section{Numerical results}\label{sec:5}
We compare our methods against two powerful dimension estimators, \texttt{DANCo}~\cite{ceruti-D} and Levina-Bickel~\cite{lb}, using a manifold library proposed in \cite{hein}. The first estimator is, arguably, the state-of-the-art for this problem, while the second one is a classical well-known estimator with great performance. To see a comparison between these and other estimators we refer the reader to \cite{ceruti-D, campad-rev}. 

Table~\ref{table:manifold_desc} presents a brief description of the manifolds included in the 
study. Additionally Table~
\ref{table:parameters} contains a list of the parameters used for each one of the estimators. We compare two error 
measures, namely the \emph{Mean Square Error} (MSE) and the \emph{Mean Percentage Error} (MPE) 
which are defined as 
\[\mathrm{MSE}(\widehat d):= \dfrac{1}{T}\sum_{i=1}^T (\widehat d_i - d_i)^2 \qquad \text{and} \qquad 
\mathrm{MPE}(\widehat d) := \dfrac{100}{T}\sum_{i=1}^T \frac{|\widehat d_i - d_i|}{d_i}  \]
where $T$ is the number of trials included in the test and $\widehat d_i$ and $d_i$ are the 
estimated 
dimension and the correct dimension of the $i$th trial, respectively. 

For each one of the aforementioned manifolds, we draw $T = 50$ random samples with  
$n = 2500$ data points and then compute the 
MSE and MPE of the following four estimators: the global basic estimator (Basic), the global basic 
estimator combined with the centers heuristic (B+H), the global kernel-based estimator (Kernel), the global 
kernel-based estimator with centers heuristic (K+H), the Levina-Bickel estimator (LB), and the 
\texttt{DANCo} estimator. Tables~\ref{table:MSE} and \ref{table:MPE} summarize the results.  
\begin{table}[h]
\centering
\caption{Library of manifolds used for benchmark, for more details consult \cite{hein}.}
\label{table:manifold_desc}
\begin{tabular}{|c|c|c|c|}
\hline
\textbf{Manifold} & $d$ & $m$ & \textbf{Description}                  \\ \hline
$M_1$             & 9                            & 10                         & Sphere ${S}^9$     \\ \hline
$M_2$             & 3                            & 5                          & Affine subspace                        \\ \hline
$M_3$             & 4                            & 6                          & Nonlinear manifold        \\ \hline
$M_4$             & 4                            & 8                          & Nonlinear manifold                                        \\ \hline
$M_5$             & 2                            & 3                          & Helix                                 \\ \hline
$M_6$             & 6                            & 36                         & Nonlinear manifold                                       \\ \hline
$M_7$             & 2                            & 3                          & Swiss roll                            \\ \hline
$M_8$             & 12                           & 72                         & Highly curved manifold \\ \hline
$M_9$             & 20                           & 20                         & Full-dimensional cube                          \\ \hline
$M_{10}$          & 9                            & 10                         & 9-dimensional cube                    \\ \hline
$M_{11}$          & 2                            & 3                          & Ten-times twisted Mobius band         \\ \hline
$M_{12}$          & 10                           & 10                         & Multivariate Gaussian                 \\ \hline
$M_{13}$          & 1                            & 10                         & Curve                 \\ \hline
\end{tabular}
\end{table}

\begin{table}[h]
\centering
\caption{Parameters of each algorithm.}
\label{table:parameters}
\begin{tabular}{|c|c|}
\hline
{\textbf{Method}} & \textbf{Parameters}                \\ \hline
ANOVA                                 & $k = 34, c = 16$ \\ \hline
LB                                    & $k_1 = 10, k_2 = 20$               \\ \hline
\texttt{DANCo}                                 & $k = 10$                           \\ \hline
\end{tabular}
\end{table}

In both these tables, the last column shows the average of the performance measure over
the examples. It is clear from these tables that the angle-variance methods introduced in the
present article do well, in terms of average performance, against the very strong competitors
considered. This is more evident when the measure of error is the MSE. Still, for many of the
manifolds considered, namely M$_1$, M$_3$, M$_4$ (in this case tied with LB), 
M$_9$ and M$_{12}$, \texttt{DANCo} clearly displays the best performance. The Levina-Bickel estimator
is the best for manifold M$_6$, tying for first with \texttt{DANCo} in M$_4$, while the
procedures proposed in this article show the best performance in the cases of manifolds
M$_8$ and M$_{10}$, having very good performance also in cases M$_3$, M$_5$, M$_7$ 
and M$_{12}$. It is interesting that our methods do particularly well in case M$_8$, 
a high curvature manifold of
a relatively high dimensional in a large dimension ambient space. 
It does not appear to exist a significant difference in performance between
the Basic procedure and the procedure that uses Kernel Density Estimation. On the other
hand, the introduction of the heuristics discussed in Section 4 turns out to be beneficial
for our estimators  in some of the relatively 
high dimensional cases, namely M$_9$ and M$_{12}$,
while these heuristics degrade somehow the performance in the intermediate dimension cases,
M$_4$ and M$_6$. In all other cases, the use of the heuristics for center selection and
center elimination does not appear to have a strong effect.

\begin{table}[]
\centering
\caption{Rounded Mean Square Error for different manifolds, last column displays the average MSE over 
all the examples. The darker cells show the best results in each column.}
\label{table:MSE}
\makebox[\textwidth][c]{
\begin{tabular}{|c|c|c|c|c|c|c|c|c|c|c|c|c|c||c|}
\hline
& $M_1$ & $M_2$ & $M_3$ & $M_4$ & $M_5$ & $M_6$ & $M_7$ & $M_8$  & $M_9$  & $M_{10}$ & $M_{11}$ & $M_{12}$ & $M_{13}$ & Mean \\ \hline
Basic                    & 0.95 & \cellcolor[HTML]{EFEFEF}0.00 & 0.58 & 0.03 & \cellcolor[HTML]{EFEFEF}0.00 & 0.67 & \cellcolor[HTML]{EFEFEF}0.00 & 1.72  & 10.39  & \cellcolor[HTML]{EFEFEF}0.00 & \cellcolor[HTML]{EFEFEF}0.00 & 0.12 & \cellcolor[HTML]{EFEFEF}0.00 & 1.11     \\ \hline
B+H      & 1.09 & \cellcolor[HTML]{EFEFEF}0.00 & 0.66 & 0.28 & \cellcolor[HTML]{EFEFEF}0.00 & 1.30 & 0.01 & 2.14  & 4.64  & 0.03 & \cellcolor[HTML]{EFEFEF}0.00 & 0.10 & \cellcolor[HTML]{EFEFEF}0.00 & \cellcolor[HTML]{EFEFEF}0.79       \\ \hline
Kernel             & 0.95 & \cellcolor[HTML]{EFEFEF}0.00 & 0.68 & 0.08 & \cellcolor[HTML]{EFEFEF}0.00 & 0.69 & 0.29 & \cellcolor[HTML]{EFEFEF}1.27  & 10.20  & \cellcolor[HTML]{EFEFEF}0.00    & \cellcolor[HTML]{EFEFEF}0.00    & 0.15    &\cellcolor[HTML]{EFEFEF}0.00    & 1.10       \\ \hline
K+H & 0.99 & \cellcolor[HTML]{EFEFEF}0.00 & 0.76 & 0.45 & \cellcolor[HTML]{EFEFEF}0.00 & 1.31 & 0.31 & 2.48  & 4.06  & 0.01    & \cellcolor[HTML]{EFEFEF}0.00    & 0.04    &\cellcolor[HTML]{EFEFEF}0.00   & 0.80       \\ \hline
LB           & 0.49 & 0.02 & 0.05 & 0.00 & 0.00 & \cellcolor[HTML]{EFEFEF}0.10 & 0.00 & 2.27  & 29.33 & 2.23 & 0.00    & 0.54    & 0.00    & 2.69       \\ \hline
\texttt{DANCo}         &\cellcolor[HTML]{EFEFEF}0.16 & \cellcolor[HTML]{EFEFEF}0.00 & \cellcolor[HTML]{EFEFEF}0.00 & \cellcolor[HTML]{EFEFEF}0.00 & \cellcolor[HTML]{EFEFEF}0.00 & 1.00 & \cellcolor[HTML]{EFEFEF}0.00 & 25.22 & \cellcolor[HTML]{EFEFEF}0.96 & 0.10 & \cellcolor[HTML]{EFEFEF}0.00    & \cellcolor[HTML]{EFEFEF}0.00    & \cellcolor[HTML]{EFEFEF}0.00    & 2.11       \\ \hline
\end{tabular}}
\end{table}

\begin{table}[]
\centering
\caption{Rounded Mean Percentage Error for different manifolds, last column displays the average MPE over all the examples.}
\label{table:MPE}
\makebox[\textwidth][c]{
\begin{tabular}{|c|c|c|c|c|c|c|c|c|c|c|c|c|c||c|}
\hline
& $M_1$ & $M_2$ & $M_3$ & $M_4$ & $M_5$ & $M_6$ & $M_7$ & $M_8$  & $M_9$  & $M_{10}$ & $M_{11}$ & $M_{12}$ & $M_{13}$ & Mean \\ \hline
Basic                & 10.56                         & \cellcolor[HTML]{EFEFEF}0.00 & 15.50                         & 1.00                         & \cellcolor[HTML]{EFEFEF}0.00 & 10.83                         & \cellcolor[HTML]{EFEFEF}0.00 & 9.00 & 15.50 & \cellcolor[HTML]{EFEFEF}0.00                         & \cellcolor[HTML]{EFEFEF}0.00 & 1.50                         & \cellcolor[HTML]{EFEFEF}0.00 & \cellcolor[HTML]{EFEFEF}4.93 \\ \hline
B+H              & 11.33                         & \cellcolor[HTML]{EFEFEF}0.00 & 17.75                         & 8.25                         & \cellcolor[HTML]{EFEFEF}0.00 & 17.83                         & 0.50                         & 9.83 & 10.00 & 0.67                         & \cellcolor[HTML]{EFEFEF}0.00 & 1.20                         & \cellcolor[HTML]{EFEFEF}0.00 & 5.95                         \\ \hline
Kernel                & 10.56                         & \cellcolor[HTML]{EFEFEF}0.00 & 17.75                         & 2.50                         & \cellcolor[HTML]{EFEFEF}0.00 & 11.50                         & 16.50 & \cellcolor[HTML]{EFEFEF}7.75 & 15.55 & \cellcolor[HTML]{EFEFEF}0.00 & \cellcolor[HTML]{EFEFEF}0.00 & 1.80                         & \cellcolor[HTML]{EFEFEF}0.00                         & 6.45                         \\ \hline
K+H              & 11.00                         & \cellcolor[HTML]{EFEFEF}0.00 & 20.00                         & 12.25                         & \cellcolor[HTML]{EFEFEF}0.00 & 18.00                         & 17.50 & 10.75 & 9.60 & 0.11                         & \cellcolor[HTML]{EFEFEF}0.00 & 0.60                         & \cellcolor[HTML]{EFEFEF}0.00                         & 7.67                         \\ \hline
LB               & 7.77                         & 4.84                         & 5.77 & 1.49                        & 1.84                         & \cellcolor[HTML]{EFEFEF}5.13                         & 2.63 & 12.55                         & 27.07 & 16.57                         & 1.64                         & 7.34                         & 0.50                         & 7.31                         \\ \hline
\texttt{DANCo} & \cellcolor[HTML]{EFEFEF}1.77 & \cellcolor[HTML]{EFEFEF}0.00 & \cellcolor[HTML]{EFEFEF}0.00 & \cellcolor[HTML]{EFEFEF}0.00 & \cellcolor[HTML]{EFEFEF}0.00 & 16.66                         & \cellcolor[HTML]{EFEFEF}0.00 & 41.83 & \cellcolor[HTML]{EFEFEF} 4.80 & 1.11                         & \cellcolor[HTML]{EFEFEF}0.00 & \cellcolor[HTML]{EFEFEF}0.00 & \cellcolor[HTML]{EFEFEF}0.00 & 5.09                         \\ \hline
\end{tabular}}
\end{table}

\end{document}